\documentstyle[twoside]{siam} 
\def\rb{{\it I\hspace{-0.3em}R}}  

\def\scs{{\cal S}}
\def\didots{\mathinner{\mkern1mu\raise1pt\vbox{\kern7pt\hbox{.}}\mkern2mu
\raise4pt\hbox{.}\mkern2mu\raise7pt\hbox{.}\mkern1mu}}
\def\dnZN{\delta^\nabla_{Z,N}}
\def\dnNZ{\delta^\nabla_{N,Z}}
\def\ddZN{\delta^\Delta_{Z,N}}
\def\ddNZ{\delta^\Delta_{N,Z}}
\def\espace1{\medskip}

\begin{document} \bibliographystyle{plain}

\newtheorem{property}[theorem]{Property}
\newtheorem{definition}[theorem]{Definition}
\newtheorem{example}[theorem]{Example}
\bibliographystyle{plain}
\title{Structured matrices and inverses \thanks{This work has been carried out
while the author was visiting the Institute for Mathematics and its
Applications, University of Minnesota, Minneapolis, in April 1992.}} 
\author{P. Comon\thanks{Thomson-Sintra, Sophia Antipolis BP138, 06561 Valbonne
Cedex France. The author can be reached at \mbox{na.comon@na-net.ornl.edu}.}}

\maketitle

\begin{abstract}
A matrix (and any associated linear system) will be referred to as structured
if it has a small displacement rank. It is known that the inverse of a
structured matrix is structured, which allows fast inversion (or solution),
and reduced storage requirements. According to two definitions of displacement
structure of practical interest, it is shown here that several types of
inverses are also structured, including  the Moore-Penrose inverse of
rank-deficient matrices.
\end{abstract}


\begin{AMSMOS}
15A03, 15A06, 15A09, 15A57, 65F20, 65F30.
\end{AMSMOS}
\bigskip

\section{Introduction}  
Close to T\"oplitz or close to Hankel matrices appear in various areas
including signal
processing and automatic control (e.g. prediction of second-order
nearly stationary time series).
In radar or sonar (or more generally antenna processing), T\"oplitz matrices
are encountered when far-field sources impinge an array of regularly spaced
sensors after propagating through an homogeneous medium. If 2-dimensional
regular arrays are utilized, then block-T\"oplitz matrices can be found. 
Other applications include optics, image processing (when the spreading
function is shift invariant), differential or integral equations under certain
boundary conditions and for certain discretizations (e.g. oil prospecting),
seismics, geophysics, transmission lines, and communications$\dots$ In
general, these applications correspond to the solution of some inverse
problems.  When shift invariance properties are satisfied, the linear operator
to invert is T\"oplitz, or block-T\"oplitz, and it is dealt with a
deconvolution problem.

However, T\"oplitz matrices in the strict sense are rarely encountered in the
real word, because the abovementioned invariance properties are not satisfied.
For instance, second-order stationarity of long time series, or homogeneity of
propagation media, are idealized assumptions.  In antenna array processing,
the decalibration of the array is the main cause of many problems, among
which the deviation from T\"oplitz is one of the mildest ones. For instance
in sonar, decalibration occurs because of the effects of pressure,
temperature, and usage, among others.  Another major cause of problems is the
distorsion of wave fronts impinging the array due to inhomogeneity of the
medium or to local turbulences (note that improvements can be obtained by
assuming that the celerity is random with a small variance, but this is out
of the scope of the present discussion).
Lastly, a simple deviation from T\"oplitz that has been already studied is the
effect of limited length of the data.  The proximity to T\"oplitz then depends
on the way the matrix is calculated:
its displacement rank ranges from 2 to 4 under ideal assumptions.

Since the set of T\"oplitz (or Hankel) matrices is a linear space,
it is easy to compute
the closest T\"oplitz approximate of any matrix by a simple projection.
However, this operation should be avoided in general, since it would destroy
other important structures (e.g.~just the rank).
On the other hand, finding the best approximate of a matrix by another of
given rank and given displacement rank is still an open problem. It is true
that some simple iterative algorithms have already been proposed in the
literature for the T\"oplitz case, but the convergence issue has not been
completely covered.

Since the early works by Schur (1917), Levinson (1947), Durbin (1960), Trench
(1964), and Bareiss (1969), a lot of work has been done. In particular,
Kailath and others introduced in the seventies the concept of displacement
rank, which allows in some way to measure a distance to T\"oplitz
\cite{KVM78}. By the way, the concept of displacement rank may be seen to have
some connections with integral and differential equations \cite{K91}. An
excellent survey of related works can be found in \cite{K87}.  Other
recent investigations are also reported in \cite{CK90}.
\par\medskip
It is known that a linear system $Tx=b$ can be solved with $O(n^2)$ flops if
$T$ is a $n \times n$ T\"oplitz matrix. If $T$ is just only close to
T\"oplitz, it is useful to define a displacement rank, $\delta$, measuring a
distance to the T\"oplitz structure \cite{FMKL79}. Then it has been shown that
the solution requires only $O(\delta n^2)$ flops, to be compared to the
$O(n^3)$ complexity necessary to solve a dense linear system of general form.
More recently, superfast algorithms have been proposed to solve T\"oplitz
systems, and their complexity ranges from $O(n log^2n)$ to $O(\alpha n log n)$,
$\alpha < n$, \cite{AG88} \cite{BBD86} \cite{CS89}. \par
The displacement rank of a linear system is clearly related to the complexity
of its solution. It has been shown in \cite{C89} \cite{CKL87} that this
complexity reduction also holds for the calculation of various
factorizations, provided the Schur algorithm is run on the appropriate
block-matrix. In this paper, the displacement rank will be defined in a
slightly more general framework, such that the structure of a wider class of
matrices can be taken into account. In this framework, the first step in the
quest of fast algorithms is to check whether the system considered has a
displacement structure, and under what displacement operator its displacement
rank is the smallest. Building explicitly fast algorithms taking advantage of
this structure is the next question. However, our investigations are limited
in this paper to the study of the displacement rank itself, and it will not
be discussed how to build the corresponding fast algorithm.
\par
The paper is organized as follows.
Definitions and basic properties are given in section 2.
In section 3, the structure of inverses and products of full-rank structured
matrices is analyzed.
Section 5 is devoted to the study of structured rank-deficient matrices, and
utilizes preliminary results derived in section 4.

\medskip
\section{Definitions and first properties} 
The structure that will be considered in this paper is exclusively the
{\em displacement structure} \cite{KVM78} \cite{FMKL79}.
Roughly speaking, a structured matrix is the sum of displaced versions of a
unique generating matrix of small rank. For instance, sparse matrices may not
have any interesting displacement structure. Displacement operators can be
defined in different ways, and two definitions will be used subsequently.
\espace1
\begin{definition}\label{def:Nabla}
For any fixed pair of matrices $(Z,N)$ of appropriate dimension, the
displacement of a matrix A with respect to displacement operator
$\nabla_{Z,N}$ is defined as
 \begin{equation}\label{def:NablaZN}
 \nabla_{Z,N} A=A-ZAN.
 \end{equation}
\end{definition}

\begin{definition}\label{def:Delta}
For any fixed pair of matrices $(Z,N)$ of appropriate dimension, the
displacement of a matrix A with respect to displacement operator
$\Delta_{Z,N}$ is defined as
 \begin{equation}\label{def:DeltaZN}
 \Delta_{Z,N} A=ZA-AN.
 \end{equation}
\end{definition}

\espace1
In the remaining, matrices $Z$ and $N$ will be referred to as {\em
displacement matrices}, and the pair $\{,Z,N\}$ to as the {\em displacement
pattern}. . Once the above definitions are assumed in the primal space, then
it is convenient to use the definitions below in the dual space, denoting by
$({}^*)$ the transposition:
\begin{eqnarray}\nonumber
 \nabla_{N,Z} (A^*) = A^* - N A^* Z,\label{def:NablaNZ} \\
 \Delta_{N,Z} (A^*) = N A^* - A^* Z.\label{def:DeltaNZ}
\end{eqnarray}
\espace1
\begin{definition}
Matrices for which any of the four displaced matrices (\ref{def:NablaZN}),
(\ref{def:DeltaZN}), (\ref{def:NablaNZ}) or (\ref{def:DeltaNZ}) has a rank
bounded by a value that does not depend on the size of $A$ will be referred to
as structured. This rank will be called the displacement rank of $A$ with
respect to the displacement operator considered, and will be denoted as
$\delta^\nabla_{Z,N}\{A\}$, $\delta^\nabla_{N,Z}\{A^*\}$,
$\delta^\Delta_{Z,N}\{A\}$, or $\delta^\Delta_{N,Z}\{A^*\}$.
\end{definition}

\espace1
This definition is consistent with \cite{CK90}. Displacement matrices $Z$ and
$N$ are usually very simple (typically formed only  of ones and zeros).
Additionally, it can be seen that the displacement operator 
(\ref{def:NablaZN}) is easily invertible as soon as {\em either $Z$ or $N$}
is nilpotent. To see this, assume that $Z^{k+1}=0$ and explicit the
displacement $\nabla_{Z,N}$ in the sum
$$
 \sum_{i=0}^k \nabla_{Z,N}\{Z^i A N^i\}.
$$
Then this expression can be seen to be nothing else than $A$ itself. For
additional details, see \cite{K87} and references therein. Note that the
results shown in this paper will not require a particular form for matrices
$Z$ and $N$ (nilpotent for instance), unless otherwise specified. Other
considerations on invertibility of displacement operators are also tackled in
\cite{C89}. In \cite{W92}, displacement operators are defined (in a manner
very similar to \cite{C89}), but displacement ranks of products or
pseudo-inverses are unfortunately not obtained explicitely. Lastly, other
displacement structures, including (\ref{def:DeltaZN}), are being
investigated by G.Heinig.

\espace1
\begin{example}\rm \label{ex:T-and-H}
Denote $S$ the so-called lower shift matrix:
\renewcommand{\arraystretch}{0.2}
\begin{equation}\label{def:S}
S = \left( \begin{array}{cccc}
 0 &  &  &  \\
 1 &\ddots & & \\
 &\ddots & \ddots & \\
 & &1 &0
\end{array}
\right).
\end{equation}\renewcommand{\arraystretch}{1}
For Hankel matrices, it is easy to check out that we have
\begin{eqnarray}
\dnZN\{H\}\leq 2, \quad \mbox{ for $(Z,N) = (S,S),$} \\
\ddZN\{H\}\leq 2, \quad \mbox{ for $(Z,N) = (S,S^*),$}
\end{eqnarray}
whereas for T\"oplitz matrices, we have
\begin{eqnarray}
\dnZN\{T\}\leq 2, \quad \mbox{ for $(Z,N) = (S,S^*),$} \\
\ddZN\{T\}\leq 2,\quad \mbox{ for $(Z,N)=(S,S).$}
\end{eqnarray}

In these particular cases, the non-zero entries of displaced matrices are
indeed contained only in one row and one column.
These four statements hold also true if matrices Z and N are permuted.
In other words,
$$
\delta^{\Delta}_{S^*,S}\{H\} \leq 2, \mbox{~and~}
\delta^{\nabla}_{S^*,S}\{T\} \leq 2.
$$
\end{example}

It turns out that the definitions \ref{def:Nabla} and \ref{def:Delta} yield
displacement ranks that are not independent to each other. We have indeed the
following

\espace1
\begin{theorem}\label{th:linkDN}
For any given matrices $Z, N, A,$ the two inequalities below hold
 \begin{eqnarray}
 \dnZN\{A\} \leq \delta^\Delta_{Z^*,N}\{A\} + \delta^\nabla_{Z,Z^*}\{I\}, \\
 \ddZN\{A\} \leq \delta^\nabla_{Z^*,N}\{A\} + \delta^\nabla_{Z,Z^*}\{I\},
 \end{eqnarray}
where $I$ denotes the identity matrix having same dimensions as $A$.
\end{theorem}

\espace1
\begin{proof}
 $\nabla_{Z,N}A=Z(Z^*A-AN)+(I-ZZ^*)A$ shows the first inequality, and
 $\Delta_{Z,N}A=Z(A-Z^*AN)-(I-ZZ^*)AN$ shows the second one.
\end{proof}

\espace1
\begin{example}\rm
If $A$ is a circulant T\"oplitz matrix, e.g.,
$$
A = \left( \begin{array}{cccc}
 a & b & c & d \\
 d & a & b & c \\
 c & d & a & b
\end{array}
\right),
$$
then it admits a displacement rank $\dnZN\{A\}=1$ provided the following
displacement pattern is assumed:
$ Z = S_3$ as given by (\ref{def:S}), and 
$$
N = \left( \begin{array}{cccc}
 0 & 1 & 0 & 0 \\
 0 & 0 & 1 & 0 \\
 0 & 0 & 0 & 1 \\
 1 & 0 & 0 & 0
\end{array}
\right).
$$
In this example, we also have $\delta^{\Delta}_{Z^*,N}\{A\} = 1$, which is
conform to theorem \ref{th:linkDN}.
\end{example}

\espace1\begin{example}\rm
Let $A$ be a $m \times n$ T\"oplitz matrix. Define $N=S_n$, and
$$Z = \left( \begin{array}{cccc}
 0 & 1 & 0 & 0 \\
 0 & 0 & \ddots & 0 \\
 1 & 0 & 0 & 1 \\
 0 & 0 & 0 & 0
\end{array} \right). $$
Then it can be seen that $\delta^\Delta_{Z^*,N}\{A\}=2$, and
$\delta^\nabla_{Z,N}\{A\}=3$. This example shows that equality can occur
in theorem \ref{th:linkDN}.
\end{example}

\espace1\begin{example}\rm
If $T$ is T\"oplitz $m \times n$ and $H$ is Hankel $m \times p$, then the
block matrix $(T~H)$ has a displacement rank equal to 3 with respect to the
displacement pattern $\{Z,N\}=\{S_m, S^*_n \oplus S_p\}$.
\end{example}
\par\espace1
The notation $A \oplus B$ will be subsequently used when $A$ and $B$ are
square to denote the block-diagonal matrix having $A$ and $B$ as diagonal
blocks.

\medskip
\section{Displacement of various inverses and products}\label{sec:various}
There is a number of situations where the displacement rank of a matrix
can be quite easily shown to be small.
Since our main concern is inverses, let us start with the simplest case.

\espace1\begin{theorem}\label{th:inverse}
Let A be an invertible square matrix. Then
 \begin{equation}\label{th:rd-inverse}
 (i)~\ddZN\{A\} = \ddNZ\{A^{-1}\},~and~(ii)~ \dnZN\{A\} = \dnNZ\{A^{-1}\}.
 \end{equation}
In other words, $A$ and $A^{-1}$ have the same displacement rank with respect
to dual displacement patterns.
\end{theorem}

\par To prove the theorem, it it useful to recall the following lemma.
\espace1\begin{lemma}\label{le:ranks}
Let $f$ and $g$ be two linear operators, and denote $E_\lambda^h$ the
eigenspace of operator $h$ associated with the eigenvalue $\lambda$. If
$\lambda$ is an eigenvalue of $fog$, then it is also an eigenvalue of $gof$.
In addition, the eigenspaces have the same dimension as soon as $\lambda$ is
non-zero:
$$ dim\{E_\lambda^{fog}\} = dim\{E_\lambda^{gof}\}. $$
\end{lemma}
\begin{proof}
Assume $\lambda$ is an eigenvalue of $fog$. Then for some non-zero vector $x$,
$fog(x) = \lambda x$. Composing both sides by operator $g$ immediately shows
that
\begin{equation}\label{eq:gof}
gof(g(x)) = \lambda g(x).
\end{equation}
Next there are two cases: (i)~if $g(x) \ne 0$,
then $g(x)$ is an eigenvector of $gof$ associated with the same eigenvalue
$\lambda$; (ii)~if $g(x) = 0$, then $fog(x) = 0$ and necessarily $\lambda = 0$.
Now assume without restricting the generality of the proof that
$dim\{E_\lambda^{fog}\} > dim\{E_\lambda^{gof}\}$. Then there exists a vector
$x$ in $E_\lambda^{fog}$ such that $g(x) = 0$ (since otherwise 
relation (\ref{eq:gof}) would imply that g(x) is also in $E_\lambda^{gof}$).
Yet composing by $f$ yields $fog(x) = 0$ and consequently $\lambda = 0$.
As a conclusion, if $\lambda \ne 0$, eigenspaces must have the same dimension.
\end{proof}

\espace1\begin{proof}{\em \hspace{-0.5em} of theorem.}
We have by definition $\ddZN\{A\} = rank\{ZA - AN\} = rank\{Z - ANA^{-1}\}$,
and $\ddNZ\{A^{-1}\} = rank\{NA^{-1} - A^{-1}Z\} = rank\{ANA^{-1} - Z\}$. But
these two matrices are opposite, and therefore have the same rank. This
proves (i).
\par Similarly since the rank does not change
by multiplication by a regular matrix, we have
$\dnZN\{A\} = rank\{A - ZAN\} = rank\{I - ZANA^{-1}\}$.
On the other hand $\dnNZ\{A^{-1}\} = rank\{A^{-1} - NA^{-1}Z\} =
rank\{I - NA^{-1}ZA\}$. Now from lemma \ref{le:ranks} we know that
$ker\{I-fog\}$ and $ker\{I-gof\}$ have the same dimension. If $f$ and $g$
are endomorphisms in the same space, this implies in particular that
$rank\{I-fog\} = rank\{I-gof\}$. Now applying this result to
$f$ = $ZA$, $g$ = $NA^{-1}$ eventually proves (ii).
\end{proof}

The proof that an invertible matrix and its inverse have the same displacement
rank has been known for a long time, and proved for symmetric matrices
\cite{K87}. However, the proof for general T\"oplitz matrices seems to have
been given only recently in \cite{CK90} for a displacement of type
(\ref{def:NablaZN}). Our theorem is slightly more general.

\espace1\begin{corollary}
For any given square matrix $A$, let the regularized inverse be given by
$R = (A + \eta I)^{-1}$,
for some number $\eta$ such that $A+\eta I$ is regular.
Then the displacement ranks of A and R are linked by the inequality below
\begin{equation}\label{cor:reg-inverse}
\delta_{N,Z}\{R\} \leq \delta_{Z,N}\{A\} + \delta_{Z,N}\{I\},
\end{equation}
this inequality holding for both displacements $\nabla$ and $\Delta$.
\end{corollary}

\espace1\begin{proof}
Just write  $\dnNZ\{R\}=\dnZN\{R^{-1}\}=\dnZN\{A+\eta I\}$, and since the
rank of a sum is smaller than the sum of the ranks, we eventually
obtain the theorem. In order to prove the inequality for the displacement
$\Delta$, proceed exaclty the same way.
\end{proof}

\espace1
When close to T\"oplitz or close to Hankel matrices are considered,
the displacement matrices $Z$ and $N$ are essentially either the
lower shift matrix $S$ or its transposed. In such a case, it is
useful to notice that
\begin{equation}\label{eq:NablaSS*_I}
\delta^\nabla_{S,S^*}\{I\} = \delta^\nabla_{S^*,S}\{I\} = 1.
\end{equation}
On the other hand for any matrix $Z$ (and $S$ or $S^*$ in particular):
\begin{equation}\label{eq:DeltaSS_I}
\delta^\Delta_{Z,Z}\{I\} = 0.
\end{equation}

\espace1
For a T\"oplitz matrix $T$, we have a stronger (and obvious) result, because
$T$ and $T+\eta I$ are both T\"oplitz.
$$
\delta^\nabla_{S^*,S}\{R\} = \delta^\nabla_{S,S^*}\{T\}.
$$

\espace1\begin{corollary}\label{cor:schur}
Let M be the $2 \times 2$ block matrix below
$$
M=\left( \begin{array}{cc}
A & B \\ C & D \end{array} \right),
$$
where A and D are square of dimension $n_1$ and $n_2$, respectively.
Assume $M$ and $A$ are invertible. When the last $n_2 \times n_2$ block of
the matrix $M^{-1}$ is invertible, it can be written as $\bar{A}^{-1}$, 
where $\bar{A}$ is the so-called Schur complement of $A$ in $M$:
$\bar{A} = D - CA^{-1}B$. If $M$ has a
displacement rank $\delta_{N,Z}\{M\}$ with respect to a 
displacement pattern $\{Z,N\}=\{Z_1 \oplus Z_2, N_1 \oplus N_2\}$, where
$Z_i$ and $N_i$ are $n_i \times n_i$ matrices, then the displacement
rank of $\bar{A}$ satisfies the inequality below for both displacements
$\nabla$ and $\Delta$:
\begin{equation}\label{eq:Schurcompl}
 \delta_{Z_2,N_2}\{\bar{A}\} \leq \delta_{Z,N}\{M\}.
\end{equation}
\end{corollary}

\espace1\begin{proof}
Applying twice the theorem \ref{th:inverse}, and noting that the rank of $M$
is always larger than the rank of any of its submatrices, yield 
$\delta_{Z_2,N_2}\{\bar{A}\}=\delta_{N_2,Z_2}\{\bar{A}^{-1}\} \leq$
$\delta_{N,Z}\{M^{-1}\}=\delta_{Z,N}\{M\}$.
\end{proof}
This kind of property has been noticed for several years by Chun and
Kailath. See for instance \cite{C89} \cite{CK90}. This corollary restates
it in the appropriate framework.

\espace1\begin{theorem}\label{th:product}
Let $A_1$ and $A_2$ be two full-rank matrices of size $n_1 \times n_2$ and
$n_2 \times n_1$, respectively, with $n_1 \leq n_2$. Then the displacement
rank of the matrix $A_1A_2$ is related to the displacement ranks of $A_1$ and
$A_2$ for either displacement $\nabla$ or $\Delta$ by
\begin{equation}
 \delta_{Z_1,Z_2}\{A_1A_2\} \leq \delta_{Z_1,N_1}\{A_1\}
 +\delta_{N_1,N_2}\{I_{n_2}\} + \delta_{N_2,Z_2}\{A_2\},
\end{equation}
\end{theorem}

\espace1\begin{proof}
To prove the theorem, form the square matrix $M$ of size $n_1+n_2$:
$$
M=\left( \begin{array}{cc}
I & A_2 \\ A_1 & 0 \end{array} \right),
$$
consider the displacement pattern $\{N_2 \oplus Z_1, N_1 \oplus Z_2\}$
and apply corollary \ref{cor:schur}. Again, since the displacement pattern is
block-diagonal, the displaced block matrix is formed of the displaced blocks.
\par In the present case, the Schur complement is precisely the product
$-A_1A_2$. This proof is identical to that already proposed in \cite{CK90} for
particular structured matrices.
\end{proof}
\espace1
Note that if $N_1=N_2$, (\ref{eq:DeltaSS_I}) implies
$\delta^\Delta_{N_1,N_2}\{I\}=0$. On the other hand, if $N_1=N_2^*=S$, then
$\delta^\nabla_{N_1,N_2}\{I\}=1$ from (\ref{eq:NablaSS*_I}).
For particular displacement matrices $Z$ and $N$, the general bounds given by
theorem \ref{th:product} may be too loose.
In particular for T\"oplitz or Hankel matrices, the corollary below is more
accurate.

\espace1\begin{corollary}
Let S be the lower shift matrix defined in (\ref{def:S}), $T_1$ and $T_2$ be
T\"oplitz matrices, and $H_1$ and $H_2$ be Hankel. Then under the conditions
of theorem \ref{th:product}:
\begin{eqnarray}
{\bf (a)} ~\delta^\Delta_{S,S}\{T_1T_2\} \leq 4, &
{\bf (b)} ~\delta^\Delta_{S,S}\{H_1H_2\} \leq 4, \label{eq:prodD} &
{\bf (c)} ~\delta^\Delta_{S,S^*}\{T_1H_2\} \leq 4, \\
{\bf (a)} ~\delta^\nabla_{S,S^*}\{T_1T_2\} \leq 4, &
{\bf (b)} ~\delta^\nabla_{S,S^*}\{H_1H_2\} \leq 4, \label{eq:prodN} &
{\bf (c)} ~\delta^\nabla_{S,S}\{T_1H_2\} \leq 4.
\end{eqnarray}
\end{corollary}

\begin{proof}
Equations (\ref{eq:prodD}) result from a combination of example
\ref{ex:T-and-H} and theorem \ref{th:product}. In fact, take $Z_i=N_i=S$ for
{\bf (a)}, $Z_1=Z_2=N_1^*=N_2^*=S$ for {\bf (b)}, and $Z_1=N_1=N_2=Z_2^*=S$
for {\bf (c)}.\par
On the other hand, if we try to apply theorem \ref{th:product} to prove
(\ref{eq:prodN}), we find a result weaker than desired, for we obtain
$\delta^\nabla \leq 5$. A more careful analysis is therefore necessary.
Restart the proof of theorem \ref{th:product}: if $T_1$ and $T_2$ are full
rank T\"oplitz, the displaced block matrix $\nabla_{S \oplus S,
S^* \oplus S^*}M$ has the following form:
$$
\left( \begin{array}{cc}
\nabla I & \nabla T_2 \\
\nabla T_1 & 0
\end{array} \right) = 
\left( ~ \mbox{\footnotesize  \tabcolsep=3pt
\begin{tabular}{|c|cc|cc|} \hline
x & ~ & ~ & x & x \\ \cline{1-1} \cline{5-5}
\multicolumn{3}{|c|}{~} & x &\multicolumn{1}{|l|}{~} \\
\multicolumn{3}{|c|}{~} & x &\multicolumn{1}{|l|}{~} \\ \hline
\multicolumn{1}{|c}{x}& x & x & ~ & ~ \\ \cline{2-3}
x & ~ & ~ & ~ & ~ \\ \hline
\end{tabular}  } ~ \right),
$$
where crosses indicate the only locations where the matrix is allowed to have
non-zero entries: only in two rows and two columns. Such a matrix is clearly
of rank at most 4. Following the same lines as in theorem \ref{th:product}, it
can be seen that the product $T_1T_2$ has a displacement rank bounded by
4.\par A similar proof could be derived in the case of two Hankel matrices,
and will not be detailed here.
In order to prove (\ref{eq:prodN}{\bf c}), let us consider finally the
block matrix
$$
M=\left( \begin{array}{cc}
I & H_2 \\ T_1 & 0 \end{array} \right).
$$
Assuming the displacement pattern $\{S \oplus S, S^* \oplus S\}$,
the displaced matrix $\nabla M$ is now of the form
$$
\left( \begin{array}{cc}
\nabla I & \nabla H_2 \\
\nabla T_1 & 0
\end{array} \right) = 
\left( ~ \mbox{\footnotesize  \tabcolsep=3pt
\begin{tabular}{|c|cc|c|c|} \hline
x & ~ & ~ & x & x \\ \cline{1-1} \cline{4-5}
\multicolumn{3}{|c|}{~} & ~ & x \\
\multicolumn{3}{|c|}{~} & ~ & x \\ \hline
\multicolumn{1}{|c}{x}& x & x & \multicolumn{2}{c|}{~} \\ \cline{2-3}
x & ~ & ~ & \multicolumn{2}{c|}{~} \\ \hline
\end{tabular}  } \right),
$$
which is again obviously of rank at most 4. The last result follows.
\end{proof}

\par The theorem \ref{th:product} was valid only for full-rank matrices of
transposed sizes. For further purposes it is useful to extend it to products
of rectangular matrices of general form.
\espace1\begin{theorem}\label{th:prod_rectang}
Let $A$ and $B$ be $m \times n$ and $n \times p$ matrices. Then the
product $AB$ is also structured, and the inequality below holds:
\begin{equation}
\delta^\Delta_{Z_A,N_B}\{AB\} \leq \delta^\Delta_{Z_A,N_A}\{A\} +
\delta^\Delta_{N_A,Z_B}\{I_n\} + \delta^\Delta_{Z_B,N_B}\{B\},
\end{equation}
where $I_n$ denotes the $n \times n$ identity matrix.
\end{theorem}

\espace1\begin{proof}
Write first the displaced matrix as
$$ \Delta_{Z_A,N_B} AB = (Z_A A - A N_A)B + A(N_A B - B N_B). $$
Then splitting the second term into $A(N_A-Z_B)B + A(Z_B B - B N_B)$ gives
\begin{equation}\label{eq:DeltaAB}
\Delta_{Z_A,N_B}(AB) = \Delta_{Z_A,N_A}A~.~B+A~.~\Delta_{N_A,Z_B}I~.~B
+A~.~\Delta_{Z_B,N_B}B,
\end{equation}
which eventually proves the theorem, since the rank of a product is always
smaller than the rank of each of its terms.
\end{proof}
A similar result holds for displacement $\nabla$. A direct consequence of
equation (\ref{eq:DeltaAB}) is the following corollary, that looks like a
differentiation rule.

\espace1\begin{corollary}\label{cor:product}
When $A$ and $B$ have dual displacement patterns, we obtain the following
simple result:
$$ \Delta_{Z,Z}(AB) = \Delta_{Z,N}A~.~B + A~.~\Delta_{N,Z}B. $$
\end{corollary}
In particular, if $AB=I$, this displaced matrix is null, because of
(\ref{eq:DeltaSS_I}).

\espace1\begin{corollary}\label{cor:invRect}
Let $A$ be a full-rank $m \times n$ matrix, with $m \geq n$. Then its
pseudo-inverse $B=(A^*A)^{-1}A^*$ has a reduced displacement rank, as show the
two bounds below:
\begin{eqnarray}
\delta^\Delta_{N,Z}\{B\} \leq \delta^\Delta_{Z,N}\{A\} +
2 \delta^\Delta_{N,Z}\{A^*\}, \label{eq:rectang-pinv} \\
\delta^\Delta_{N,Z^*}\{B\} \leq 3 \delta^\Delta_{Z,N}\{A\} +
\delta^\Delta_{Z^*,Z}\{I_m\}.
\end{eqnarray}
\end{corollary}
\begin{proof}
Apply corollary \ref{cor:product} to $A^*A$, next theorem \ref{th:inverse},
and lastly theorem \ref{th:prod_rectang}.
\end{proof}

\espace1\begin{example}\rm
If $A$ is T\"oplitz, equation (\ref{eq:rectang-pinv}) claims that
$\delta^\Delta_{S,S}\{B\} \leq 6$. In practice, it seems that no T\"oplitz
matrix could yield a displacement rank larger than
$\delta^\Delta_{S,S}\{B\} = 4$, which suggests that the bound is much too
large.
\end{example}

\espace1\begin{definition}
Given any matrix $A$, if a matrix $A^-$ satisfies
\begin{eqnarray}
(i)~ AA^- A = A, & (ii)~ A^-AA^- = A^-, \nonumber \\
(iii)~ (AA^-)^* = AA^-, & (iv)~ (A^-A)^* = A^- A, \label{def:A-}
\end{eqnarray}
then it will be called the Moore-Penrose (MP) pseudo-inverse of $A$.
A so-called generalized inverse need only to satisfy conditions (i) and (ii).
\end{definition}

\espace1
It is well known that $A^-$ is unique, and that $A^-$ and $A^*$ have the same
range and the same null space \cite{GV83}. On the other hand, a generalized
inverse is not unique.
When a matrix $A$ is rank deficient, it is in general not possible to
construct a MP pseudo-inverse  having the same displacement rank, as will be
demonstrated in section \ref{sec:MP}.
\medskip
\section{The space of $P$-symmetric matrices}
In this section, more specific properties shared by matrices in a wide class
will be investigated. The property of $P$-symmetry will be necessary in
section \ref{sec:MP} to transform a matrix into its transposed just by a
congruent transformation.

\espace1\begin{definition}
Let $P$ be a fixed orthogonal $n$ by $n$
matrix.  The set of $P$-symmetric matrices is defined as follows:
\begin{equation}\label{def:SP}
\scs_P = \{ M\in \rb^{n\times n} / PMP^* = M^*\},
\end{equation}
where $({}^*)$ denotes transposition and $\rb$ the set of real numbers.
\end{definition}

\espace1
It will be assumed in this section that the matrix to invert (or the system to
solve) belongs to $\scs_P$, for some given known orthogonal matrix $P$.  For
instance, if a matrix $A$ is square and T\"oplitz, then it is centro-symmetric
and satisfies
$$
JAJ^* = A^*,
$$
which shows that $A\in\scs_J$, where $J$ denotes the reverse identity:
\begin{equation}\label{def:J}
J = \left( \begin{array}{ccc}
 &&1\\
 & \didots & \\
 1 & &
\end{array} \right).
\end{equation}
If $A$ is Hankel, then $A\in \scs_I$ because $A$ is symmetric.  The property
of $P$-symmetry is interesting for it is preserved under many transformations.
For instance, singular vectors of a $P$-symmetric matrix are $P$-symmetric in
the sense that if $\{u,v,\sigma\}$ is a singular triplet, then so is $\{Pv,
Pu, \sigma\}$. A sum or a product of $P$-symmetric matrices is $P$-symmetric.

\espace1\begin{example}\rm\label{ex:alterTopl}
Define the `alternate T\"oplitz matrix' below
$$ A = \left( \mbox{\small\tabcolsep=3pt\begin{tabular}{rrrrr}
2 &-2 &-2 & 1 & 1 \\
1 &-2 &-2 & 2 & 1 \\
1 & 1 & 2 &-2 &-2 \\
4 &-1 & 1 &-2 &-2 \\
-8& 4 & 1 & 1 & 2 \end{tabular} } \right), $$
and assume the displacement pattern
$$ Z = \left( \mbox{\small\tabcolsep=3pt\begin{tabular}{rrrrr}
0 & 0 & 0 & 0 & 0 \\
1 & 0 & 0 & 0 & 0 \\
0 &-1 & 0 & 0 & 0 \\
0 & 0 & 1 & 0 & 0 \\
0 & 0 & 0 &-1 & 0 \end{tabular} } \right),
{\rm ~and~} N = - Z^*.$$
Then we have $P A P^* = A^*$ as requested in the definition above, with $P=J$.
This matrix has displacement ranks $\dnZN\{A\} = 2$ and
$\delta^\Delta_{Z^*,N}\{A\} = 2$, and is singular. The displacement rank of
its MP pseudo inverse will be calculated in example \ref{ex:pinv}.
\end{example}

\espace1
\begin{property}
The properties of $P$-symmetry and $P^*$-symmetry are equivalent.
\end{property}

\espace1
\begin{proof}  Let $A$ be $P$-symmetric.  Then transposing (\ref{def:SP})
gives $M=PM^*P^*$.  Next pre- and post-multiplication by $P^*$ and $P$,
respectively, yields $P^*MP=M^*$.
\end{proof}
\espace1
\begin{theorem}\label{th:P*-sym}
If $A$ is $P$-symmetric, then so is $A^{-1}$ whenever $A$ is invertible. If
$A$ is singular, then its Moore-Penrose inverse, $A^-$, is also $P$-symmetric.
\end{theorem}

\espace1 \begin{proof}
Inversion of both sides of the relation $PAP^* = A^*$ yields immediately
$PA^{-1} P^* = A^{-1*}$.
Now to insure that when $A$ is singular, $A^-$ is $P$-symmetric,
it suffices to prove that the matrix $B=PA^{-*} P^*$
indeed satisfies the four conditions of definition
(\ref{def:A-}). First, $ABA=APA^{-*}P^*A$ yields
$ABA=PA^* A^{-*} A^* P^* = PA^* P=A$, which shows
(i) of (\ref{def:A-}).  Second, $BAB=PA^{-*} P^* APA^{-*} P^*$ yields
similarly $BAB =P A^{-*} A^* A^{-*} P^* = PA^{-*} P^*$,
which equals $B$ by definition.  Next to prove (iii), consider
$AB=APA^{-*} P^*$, which gives after premultiplication by $PP^*$:
$AB=PA^*A^{-*}P^*$. But since $(A^-A)^*=A^-A$, we have $AB=PA^-AP^*$. Then
insertion of $P^*P$ yields finally
$AB=PA^-P^*A^*$, which is nothing else then $B^*A^*$.
The proof of (iv) can be derived in a similar manner.\end{proof}

It may be seen that in the last proof, $A$ does not need to be a normal
matrix, which was requested in a similar statement in \cite{HBW90}.
On the other hand, it is true that if $A$ is
$P$-symmetric, $AA^*$ is in general not $P$-symmetric.

\medskip
\section{Displacement of MP pseudo-inverses}\label{sec:MP}
In section \ref{sec:various}, it has been shown among other things that the
pseudo-inverse of a full-rank matrix is structured. It will be now analyzed
how the rank deficience weakens the structure of the MP pseudo-inverse.

\espace1\begin{theorem}\label{th:pinv}
Let $A$ be a P-symmetric square matrix, and let $Z$ and $N$ be two
displacement matrices linked by the relation
\begin{equation}\label{eq:relationNZ}
PZP = N.
\end{equation}
Then the displacement ranks of $A$ and $A^-$ are related by
\begin{equation}
 \dnNZ\{A^-\} \leq 2 \dnZN\{A\}.
\end{equation}
\end{theorem}

\espace1
In this theorem, the condition (\ref{eq:relationNZ}) is satisfied in
particular for both close to T\"oplitz and close to Hankel matrices, with
$(P,Z,N)=(J,S,S^*)$ and $(P,Z,N)=(I,S,S)$, respectively.

\espace1\begin{proof}
For conciseness, denote in short $\delta$ the displacement rank $\dnZN\{A\}$,
and assume A is $n \times n$.
In order to prove the theorem, it is sufficient to find two full-rank
$n \times n-\delta$ matrices $E_1$ and $E_2$ such that
\begin{equation}\label{eq:E2DE1=zero}
E_2^* \nabla_{N,Z}\{A^-\} E_1 = 0.
\end{equation}
For this purpose, define the following full-rank matrices with $n$ rows:
\begin{equation}\label{eq:sets-def}
\begin{array}{rcl}
G_1 = & matrix~whose~columns~span & Ker \nabla A \\
G_2 = & matrix~whose~columns~span & Ker (\nabla A)^* \\
K_1 = & matrix~whose~columns~span & Ker AN \cap Ker \nabla A \\
K_2 = & matrix~whose~columns~span & Ker (ZA)^* \cap Ker (\nabla A)^* \\
V_1 = & matrix~whose~columns~span & A N G_1 \\
V_2 = & matrix~whose~columns~span & A^* Z^* G_2.
\end{array}
\end{equation}\label{eq:Ei-def}
Then define the two matrices $E_i$ as:
$$
E_i = [ V_i,~ W_i ], with~ W_i = P K_i.
$$
Let us prove first that $E_i$ are indeed of rank $n-\delta$, and
then that (\ref{eq:E2DE1=zero}) is satisfied.\par
From (\ref{eq:sets-def}), we have by construction $A K_1 = 0$.
Then inserting a factor $P^*P$
and premultiplying by P gives $P A P^* P K_1 = 0$, which shows that 
$A^* W_1 = 0$. Yet, $V_1$ is in the range of $A$ by definition, and thus
$V_1$ and $W_1$ are necessarily orthogonal as members of the orthogonal
subspaces $KerA^*$ and $ImA$. In addition, $P$ is bijective so that $W_1$
has the same dimension as $K_1$. As a consequence, $dimE_1= dimV_1 + dimK_1$,
which is nothing else but $dimG_1$ if we look at the definitions
(\ref{eq:sets-def}).
Similarly, one can show that  $W_2$ and $V_2$ are orthogonal because 
$W_2 \subset KerA$. This yields after the same argumentation that
$dimE_2=dimG_2=n-\delta$.\par
Now it remains to prove (\ref{eq:E2DE1=zero}). To do this, it is shown that
the four blocks of $E_2^* \nabla A^- E_1$ are zero.
The quantity $\mu=V_2^*\nabla A^-V_1$ is null since $\mu=G_2^*ZA(A^- -
NA^-Z)ANG_1$ can be written $\mu=G_2^*ZANG_1-G_2^*ZANA^-ZANG_1$, which is the
difference of two identical terms by construction of matrices $G_i$.
In fact from (\ref{eq:sets-def}), $ZANG_1 = AG_1$ and $G_2^*ZAN = G_2^*A$.
Next $W_2^* \nabla A^-$ is null because $W_2^* \nabla A^*$ is null
(remember that $A^-$ and $A^*$ have the same null space). In fact,
$W_2^* \nabla A^* = K_2^*P^*A^*-K_2^*P^*NA^*Z$ by definition of $W_2$ and
$\nabla$. Now using the relation
(\ref{eq:relationNZ}) and $P$-symmetry of A yield
$W_2^* \nabla A^*P = K_2^*A - K_2^*ZAN$.
These two terms are eventually null by construction of $K_2$.
It can be proved in a similar manner that $\nabla A^* W_1=0$. In fact,
$\nabla A^* W_1=A^*PK_1-NA^*ZPK_1$ implies
$P^*\nabla A^* W_1= AK_1 - P^*NP^*PA^*P^* PZP K_1$.
Again these two terms can be seen to be zero utilizing (\ref{eq:relationNZ}),
$P$-symmetry of A, and the definition (\ref{eq:sets-def}) of $K_1$.
\end{proof}
This theorem is an extension of a result first proved in \cite{CL90}. As
pointed out in \cite{C92}, when the displacement rank of $A$ is larger than
its rank, the theorem above gives too weak results as is next shown.

\espace1\begin{theorem}\label{th:rank_pinv}
Let $A$ be a square matrix, and denote by $r\{A\}$ its rank. Then there exist 
two other bounds for the displacement rank of its MP pseudo-inverse:
\begin{eqnarray}
 \dnNZ\{A^-\} < 2 r\{A\} ~~if~ \dnZN\{A\} < 2 r\{A\}, {\rm ~and}
\label{eq:pinv2}\\
 \dnNZ\{A^-\} \leq 2 r\{A\} {\rm ~otherwise}. \label{eq:pinv3}
\end{eqnarray}
\end{theorem}

\espace1\begin{proof}
The proof of (\ref{eq:pinv3}) is easy. In fact, it holds true for any matrix
$M$ since
$$rank\{M - Z M N\} \leq rank\{M\} + rank\{Z M N\} \leq 2 rank\{M\}$$
is always true. So let us prove (\ref{eq:pinv2}). Since $A$ has rank $r\{A\}$,
it may be written as
$$ A = U \Sigma V^*, $$
where $\Sigma$ is invertible and of size $r\{A\}$. Define the matrices
$A = [U ~ ZU]$, $B = [V ~ N^*V]$, and $\Lambda = Diag(\Sigma, -\Sigma)$.
Then it may be seen that
$$ \nabla A = A \Lambda B^*. $$
Since $\Lambda$ is of full rank, either $A$ or $B$ must be rank defficient,
otherwise $\nabla A$ would be of rank $2 rank\{A\}$ which is contrary to the
hypothesis. Thus assume without restricting the generality of the proof that
$rank\{A\} < 2 r\{A\}$. Then $rank\{\nabla A^-\} < 2 r\{A\}$ because:
$$ \nabla A^{- *} = A \Lambda^{-1} B^*. $$
This completes the proof.
\end{proof}

\espace1\begin{corollary}\label{cor:pinvHT}
Let $T$ and $H$ be close to T\"oplitz and close to Hankel square matrices,
respectively. Then 
\begin{eqnarray}
  \delta^\nabla_{S^*,S}\{T^-\} \leq 2  \delta^\nabla_{S,S^*}\{T\}, 
~~\delta^\nabla_{S,S}\{H^-\}  \leq 2  \delta^\nabla_{S,S}\{H\},
\label{eq:pinv-Nabla} \\
  \delta^\Delta_{S,S}\{T^-\}  \leq 2  \delta^\Delta_{S,S}\{T\}+1,
~~\delta^\Delta_{S^*,S}\{H^-\} \leq 2  \delta^\Delta_{S,S^*}\{H\}+1.
\label{eq:pinv-Delta}
\end{eqnarray}
\end{corollary}

\espace1\begin{proof}
To prove (\ref{eq:pinv-Nabla}), simply use theorem \ref{th:pinv} and relation
(\ref{eq:NablaSS*_I}). In order to prove equations (\ref{eq:pinv-Delta}),
utilize theorem \ref{th:linkDN} and relation (\ref{eq:DeltaSS_I}).
\end{proof}
Note that the bounds are tight enough to be reached, as now shown in
examples.

\espace1\begin{example}\rm\label{ex:pinv}
Take again the matrix defined in example \ref{ex:alterTopl}.
This matrix is of rank 4 and displacement rank 2. In addition,
the displacement pattern satisfies $PZP=N$ as required in the
theorem \ref{th:pinv}. With the notations of
example \ref{ex:alterTopl}, the MP pseudo inverse of $A$
has displacement ranks $\delta^\Delta_{N,Z^*}\{A^-\} = 4$ and
$\delta^\nabla_{N,Z}\{A^-\} = 4$. This is consistent with theorem
\ref{th:pinv}.
\end{example}

\espace1\begin{example}\rm
Define the $5 \times 5$ T\"oplitz matrix of rank 3 :
$$ A = \left( \mbox{\small\tabcolsep=3pt\begin{tabular}{rrrrr}
2 & 4 & 3 & 1 & 2 \\
1 & 2 & 4 & 3 & 1 \\
3 & 1 & 2 & 4 & 3 \\
4 & 3 & 1 & 2 & 4 \\
2 & 4 & 3 & 1 & 2 \end{tabular} } \right), $$
and assume as displacement pattern $Z=S_5$ and $N=S_5^*$. Then $A$ has a
displacement rank $\delta_{Z,N}^\nabla=2$, and its MP pseudo-inverse has a
displacement rank $\delta_{N,Z}^\nabla$ equal to 4. This result was expected,
according to corollary \ref{cor:pinvHT}.
\end{example}

\espace1\begin{example}\rm
If $H$ is Hankel, then the displacement rank of $H^-$ with 
respect to the displacement operator $\Delta_{S^*,S}$ is bounded by 5.
\end{example}

\espace1
Other particular examples can be found in \cite{CL90} and \cite{C92}.
Let us now switch to the case of rectangular and rank-deficient structured
matrices. In order to extend corollary \ref{cor:invRect}, we need a
variant of the inversion lemma:

\espace1\begin{lemma}
Let $M$ be the block matrix:
$$ M = \left( \begin{array}{cc} P&A_2\\A_1&0 \end{array}\right), $$
where $P$ is square invertible, and where $A_1$ and $A_2$ have the same rank.
Then the MP-pseudo-inverse of $M$ is:
$$ M^- = \left( \begin{array}{cc} 
Y & -P^{-1}A_2X \\ -XA_1P^{-1} & X 
\end{array}\right), $$
where $X=-(A_1P^{-1}A_2)^-$, and $Y=P^{-1}+P^{-1}A_2XA_1P^{-1}$.
\end{lemma}
\espace1\begin{proof}
Let $A_i=U_iD_iV_i^*$ denote the SVD of $A_i$. Then define the matrix
$$\bar{M} = \left( \begin{array}{cc} U_2^*&0\\0&U_1^* \end{array} \right)
~M~ \left( \begin{array}{cc} V_1&0\\0&V_2 \end{array} \right), $$
and apply the usual inversion lemma to the invertible square portion of
$\bar{M}$, denoted $B$. In other words we have:
$$\bar{M} = \left( \begin{array}{cc} B&0\\0&0\end{array} \right)
{\rm ~ and ~}
\bar{M}^- = \left( \begin{array}{cc} B^{-1}&0\\0&0\end{array} \right).$$
The last lines of the proof are then just obvious manipulations.
\end{proof}

\espace1\begin{corollary}\label{cor:pinvRect}
Let $A$ be an $m \times n$ rectangular matrix with $m > n$.
Then the displacement rank of its MP pseudo inverse verifies 
$$
 \delta^\Delta_{N,Z}\{A^-\} \leq 3 \delta^\Delta_{N,Z}\{A^*\} +
2 \delta^\Delta_{Z,N}\{A\} + 2 \delta^\nabla_{N,N^*}\{I_m\}.
$$
\end{corollary}

\espace1\begin{proof}
Write $A^-$ as $(A^*A)^-A^*$, apply theorem \ref{th:pinv} to the
square matrix $(A^*A)$, and then apply the product rule given in
corollary \ref{cor:invRect}.
\end{proof}

\medskip
\section{Concluding remarks}
In this paper various aspects of the displacement rank concept were addressed
in a rather general framework. In particular, displacement properties of
rank-deficient matrices were investigated. However the bounds given in
corollaries \ref{cor:invRect} and \ref{cor:pinvRect} are obviously too large.
It is suspected that corollary \ref{cor:pinvRect} could be improved to
$\delta\{B\} \leq 2\delta\{A\} + \delta\{I\}$ in most cases. On the other
hand, particular examples have been found showing that the bounds given in
other theorems are indeed reached (in particular theorems \ref{th:pinv} and
\ref{th:rank_pinv}).
\par Another major limitation of this work lies in the fact that our proofs
are in general not constructive, in the sense that they do not define suitable
algorithms having the expected complexity. This is now the next question to
answer. First ideas in this direction can be found in \cite{C89} and
\cite{HR84} and could be used for this purpose.
\par The author thanks Georg Heinig for his proofreading of the paper.
\medskip

\medskip

\rule{2cm}{0.4mm}

Published in: {\em Linear Algebra for Signal Processing}, A.~Bojanczyk
and G.~Cybenko editors, vol.69, {\em IMA volumes in Mathematics and
its Applications}, pp.1-16, Springer Verlag, 1995.
\enddocument